\documentclass[11pt]{amsart}

\usepackage{graphicx}
\usepackage{amsthm}
\usepackage{amsmath, amssymb}
\swapnumbers

\theoremstyle{plain}
\newtheorem{Thm}{Theorem}

\newtheorem{Coro}[Thm]{Corollary}
\newtheorem{Lem}[Thm]{Lemma}

\theoremstyle{definition}

\begin{document}

\title[Heegaard splittings of surface bundles]{Mapping class groups of Heegaard splittings of surface bundles}

\author{Jesse Johnson}
\address{\hskip-\parindent
        Department of Mathematics \\
        Oklahoma State University \\
        Stillwater, OK 74078 \\
        USA}
\email{jjohnson@math.okstate.edu}

\subjclass{Primary 57M}
\keywords{Heegaard splitting, mapping class group, surface bundle}

\thanks{This project was supported by NSF Grant DMS-1006369}

\begin{abstract}
Every surface bundle with genus $g$ fiber has a canonical Heegaard splitting of genus $2g+1$. We classify the mapping class groups of such Heegaard splittings in the case when the surface bundle has a sufficiently complicated monodromy map.
\end{abstract}

\maketitle

Let $M$ be a closed surface bundle with bundle map $\pi : M \rightarrow S^1$. Equivelently, we can write $M = F' \times [0,1] / \phi$ where $F'$ is a closed surface of genus $g$ and $\phi : F' \times \{0\} \rightarrow F' \times \{1\}$ is the \textit{monodromy} of $\pi$. Assume $g \geq 2$. The complexity of the monodromy map can be measured by its displacement distance $d(\phi)$ in the curve complex for $F'$, whose definition we review below.

Every surface bundle $M$ has a Heegaard surface $\Sigma$ constructed as follows: Let $F = F_1 \cup F_2 = \pi^{-1}\{0,\frac{1}{2}\}$ be two fibers of the bundle. Let $\alpha_-$, $\alpha_+$ be vertical arcs with disjoint endpoints in $F_1$, $F_2$, one in each component of $M \setminus F$. Construct $\Sigma$ by attaching disjoint tubes to $F$ along these arcs, as in Figure~\ref{surfacefig}. 

The \textit{mapping class group} $Mod(M, \Sigma)$ is the group of automorphisms of $M$ that take $\Sigma$ onto itself, modulo isotopies of $M$ that preserve $\Sigma$ setwise. The \textit{isotopy subgroup} of $Mod(M, \Sigma)$ is the set of elements that are isotopy trivial as automorphisms of $M$. In this short note, we define a subgroup $\mathcal{B} \subset Isot(M, \Sigma)$ for the standard Heegaard splitting of any surface bundle, then prove that for sufficiently high distance surface bundles, this subgroup is the entire isotopy subgroup of the mapping class group.

\begin{Thm}
\label{mainthm}
If the monodromy $\phi$ of $\pi$ has displacement $d(\phi) > 28$ then $Isot(M, \Sigma) = \mathcal{B}$.
\end{Thm}

The entire mapping class group $Mod(M, \Sigma)$ can often be deduced from $Isot(M, \Sigma)$ given further information about the automorphisms of $M$ and its different genus $2g+1$ Heegaard splittings. However, that is beyond the scope of the present paper. 

Note that the bound on $d(\phi)$ is independent of the genus of $F$. Bachman-Schleimer~\cite{bs:bundles} showed that if $g$ is the genus of $F$ and $d(\phi) > 4g$ then $\Sigma$ is a minimal genus Heegaard surface for $M$. However, for $g > 7$, it is conceivable that $\Sigma$ is not the smallest Heegaard splitting for $M$ (though we will see that it is unstabilized). 

To define the displacement $d(\phi)$, recall that the \textit{curve complex} $\mathcal{C}(F)$ for $F$ is the simplicial complex whose vertices are isotopy classes of simple closed curves in $F$ and whose simplices span sets of vertices representing pairwise disjoint curves. The map $\phi$ defines an isometry of $\mathcal{C}(F)$, which we will also denote by $\phi$. The \textit{displacement} of $\phi$ is $d(\phi) = \min\{d(v,\phi(v))\}$ where $d(\cdot,\cdot)$ is the edge-path distance between two vertices in the curve complex.

There is a pair of disjoint, weak reducing disks $D^-$, $D^+$ for $\Sigma$ dual to the arcs $\alpha_-$, $\alpha_+$, respectively, such that compressing along these disks recovers $F$. Note that any isotopy of $\alpha_{\pm}$ that keeps the arcs vertical with their endpoints disjoint and in $F$ and returns to the original arcs will define an element of $Isot(M, \Sigma)$. We will call the subgroup of all such automorphisms the \textit{arc subgroup} $\mathcal{A} \subset Isot(M, \Sigma)$.

The movement of the endpoints of $\alpha_1$, $\alpha_2$ in $F_1$ under such an isotopy defines a braid in $F_1$, and thus determines a homomorphism $b : \mathcal{A} \rightarrow B_{2, F_1}$, where $B_{2, F_1}$ is the pure braid group of two strands in $F_1$. There is a second homomorphism from the kernel of $b$ into $B_{2, F_2}$. In fact, the reader can check that this homomorphism is onto the subgroup of $B_{2, F_2}$ in which the path of each strand is homotopy trivial. After conjugating by an isotopy that takes the path of one strand to the trivial path, we can identify this with the kernel of the inclusion map from the fundamental group of a once-punctured surface to a closed surface. We will show below that this homomorphism is injective, so $\mathcal{A}$ is finitely generated and we can explicitly construct a generating set from the above description. Note that all the groups described above are finitely generated.

The bundle structure defines an infinite cyclic cover $\hat M$ of $M$ such that $\hat M$ is homeomorphic to $F' \times \mathbf{R}$ and $\pi : M \rightarrow S^1$ lifts to the projection map of $\hat M$ onto the $\mathbf{R}$ factor of the product. We can assume the product structure is such that translating along the $\mathbf{R}$ factor permutes the preimages in $\hat M$ of points in $M$ and thus defines an automorphism of $M$. All these automorphisms are isotopy trivial on $M$ and translation by any half integer defines an automorphism $r_t : M \rightarrow M$ that takes $F$ onto itself. (The automorphism will switch $F_1$ and $F_2$ whenever $t$ is not a whole integer.) Because the arcs $\alpha_-$ and $\alpha_+$ are contained in disjoint balls, we can choose the product structure on $\hat M$ so that the automorphism induced by each half integer translation permutes $\alpha_-$ and $\alpha_+$ and, moreover, takes $\Sigma$ onto itself. 

Note that this $r_t$ is not unique -- it is only defined up to composition with elements of $\mathcal{A}$. The different possible choices of $r_t$ for a given $t$ thus define a coset of $\mathcal{A}$. Define $\mathcal{B}$ as the union of all such cosets for all half integers. By construction, this is an extension of $\mathcal{A}$ by $\mathbf{Z}$. The subgroup $\mathcal{A}$ is normal in $\mathcal{B}$, so we have a semi-direct product $\mathcal{B} = \mathcal{A} \rtimes \mathbf{Z}$, where $\mathbf{Z}$ acts on $\mathcal{A}$ by conjugating by a power of the monodromy. 

The subgroup $\mathcal{A}$ consists entirely of reducible automorphisms of $\Sigma$. However, it does not fall into the classification of reducible automorphisms of strongly irreducible Heegaard splittings recently given by the author and Hyam Rubinstein~\cite{jr:mcgs} because the Heegaard surface is weakly reducible. In fact, these mapping class groups combine behavior from two of the strongly irreducible classes: The automorphisms defined by the monodromy of the bundle are similar to the cyclic automorphisms induced by open book decompositions, which are further explored in~\cite{me:openbooks}. The braid group elements, however, are reminiscent of the automorphisms induced by one-sided Heegaard splittings~\cite{me:onesided}. The present paper, however, uses completely different (and much more concise) techniques than those in~\cite{me:openbooks} and~\cite{me:onesided}, which also lead to the genus independent distance bound.

Recall that a Heegaard surface $\Sigma$ is \textit{weakly reducible} if there is a pair of disjoint compressing disks on opposite sides of $\Sigma$. These disks are called a \textit{weak reducing pair}. In particular, the disks $D^-$, $D^+$ constructed above are a weak reducing pair for $\Sigma$. The proof of Theorem~\ref{mainthm} is based on the following Lemma, which is a result of Masur-Schleimer's work on holes in the complex of curves:

\begin{Lem}
\label{mainlem}
If the monodromy of $\pi$ has displacement distance strictly greater than 28 then $D^-$, $D^+$ is the unique (up to isotopy) pair of weak reducing disks for $\Sigma$.
\end{Lem}

\begin{proof}
When we attach tubes to construct $\Sigma$ from $F$, we remove two disks from each component of the surface, then attach annuli to the boundaries where these disks were removed. Let $P^\pm_1$ be the disks removed from $F_1$ and $P^\pm_2$ the disks removed from $F_2$ such that the annulus $A^+$ along $\alpha_+$ is attached to $\partial P^+_1$ and $\partial P^+_2$, while the annulus $A^-$ along $\alpha_-$ is attached to $\partial P^-_1$ and $\partial P^-_2$.

As above, we let $D^-$, $D^+$ be the weak reducing disks dual to $\alpha_-$, $\alpha_+$, respectively, shown in Figure~\ref{surfacefig}. Assume for contradiction there is a second, distinct weak reducing pair $E^-$, $E^+$, both with non-separating boundaries, labeled so that $E^-$ is on the same side of $\Sigma$ as $D^-$. Then one of the disks $E^-$, $E^+$ is distinct from the corresponding $D^\pm$ and we will assume without loss of generality that $E^-$ is not isotopic to $D^-$. 
\begin{figure}[htb]
  \begin{center}
  \includegraphics[width=3in]{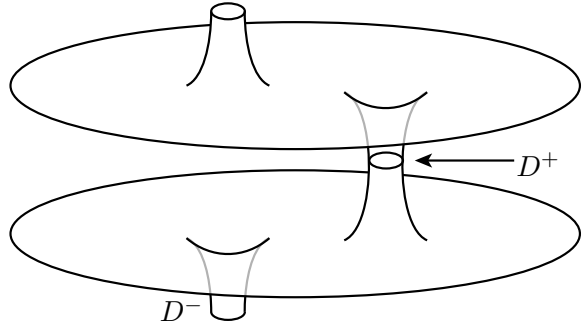}
  \put(-25,55){$D^+$}
  \put(-160,0){$D^-$}
  \caption{The genus $2g+1$ Heegaard surface.}
  \label{surfacefig}
  \end{center}
\end{figure}

Isotope $E^-$ to intersect $D^-$ minimally. If we remove the disks $P^+_1$, $P^+_2$ from $F_1$ and $F_2$, then attach the annulus $A^+$, the resulting surface bounds a handlebody $B^-$. If $E^-$ is contained in $B^-$ then let $C^- = E^-$. Otherwise, the complement $E^- \setminus B^-$ is a regular neighborhood of $E^- \cap D^-$ and we will let $C^- \subset E^-$ be an outermost disk of $E^- \cap B^-$. 

Note that $\partial C'$ intersects $P^-_1 \cup P^-_2$ in at most one arc. If $\partial C^-$ is trivial in $\partial B^-$ then it bounds a disk in $\partial B^-$ containing $P^-_1$ or $P^-_2$ in its interior. If $\partial C^-$ is isotopic to $\partial P^-_1$ or $\partial P^-_2$ in the complement of the two disks then $E^-$ is isotopic to $D^-$, contradicting our initial assumption. If $\partial C^-$ bounds a disk containing both $P^-_1$ and $P^-_2$ in its interior then $\partial E^-$ is separating in $\Sigma$, which also contradicts our initial assumption. Thus $C^-$ must intersect $P^-_1$ or $P^-_2$, and we will temporarily assume without loss of generality that it intersects $P^-_1$. This also implies that $E^-$ must intersect $\partial D^+$. Since $E^+$ is disjoint from $E^-$, we conclude that $E^+$ is not isotopic to $D^+$. 

If $P^-_2$ is in the interior of a disk in $\partial B^-$ bounded by $\partial C^-$ then every arc of $\partial E^- \cap \partial B^-$ with one endpoint in $\partial P^-_2$ must have its second endpoint in $\partial P^-_1$. This implies that $\partial E^- \cap P^-_1$ contains at least two more points than $\partial E^- \cap P^-_2$, contradicting the fact that both intersections contain the same number of points as $\partial E^- \cap \partial D^-$. We therefore conclude that $\partial C^-$ is essential in $\partial B^-$.

The intersection $\partial C^- \cap F_1$ consists of arcs properly embedded in $F_1 \setminus P^+_1$, possibly with one arc intersecting $P^-_1$ in a single subarc. We will think of the arcs as forming a graph $G^-_1$ with vertices $P^+_1$, $P^-_1$, and note that the vertex corresponding to $P^-_1$ has valence zero or two.

If $G^-_1$ is contained in a disk then there is an isotopy of $\partial C^-$ in $\partial B^-$ into $\partial B^- \cap F_2$. However, $F_2$ is incompressible in $M$ and $M$ is irreducible, so this implies that $C^-$ is parallel into $F_2$. Because $C^-$ is disjoint from the annulus $A^+$, it must be parallel into $\partial B^-$ as well, which we ruled out above. Thus the graph $G^-_1$ must contain an essential edge loop $\ell^-_1$. Since $G^-_1$ has two vertices, the edge length of $\ell^-_1$ is at most two. Similarly, we can find edge loops $\ell^+_1$ in the graphs defined by $E^+ \cap F_1$ and $\ell^\pm_2$ in $E^\pm \cap F_2$.

Because $E^-$ and $E^+$ are disjoint, any intersections between $\ell^+_1$ and $\ell^-_1$ are contained in $P^\pm_1$, so the two loops intersect in at most two points. This implies there is an essential loop $\ell_1$ disjoint from both. Similarly, there is an essential loop $\ell_2$ disjoint from $\ell^+_2$ and $\ell^-_2$.

The handlebody $B^-$ inherits an interval bundle structure, which defines a map $\gamma^+ : (F_1 \setminus P^+_1) \rightarrow (F_2 \setminus P^+_2)$.  By Masur-Schleimer~\cite[Lemma 12.12]{masur-schleimer}, for any arcs $\beta_1$ of $C^- \cap F_1$ and $\beta_2$ of $C^- \cap F_2$, we have $d(\gamma^-(\beta_1),\beta_2) \leq 6$, where the distance is in the arc complex for $F_2 \setminus P^+_2$. (The arc complex is similar to the complex of curves, but with vertices representing properly embedded, essential arcs in a surface with boundary.) Because $F^2 \setminus P^+_2$ has a single puncture, every essential arc determines an essential loop in $F^2$. Disjoint arcs define essential loops that intersect in at most one point in $F_2$, and thus have distance at most two in the curve complex $\mathcal{C}(F)$. Thus the distance from $\gamma^-(\ell^-_1)$ to $\ell^-_2$ is at most 12. 

A similar argument applies to the map $\gamma^+ : (F_2 \setminus P^-_2) \rightarrow (F_1 \setminus P^-_1)$ defined by the surface that results from only attaching $A^-$ to $F$. The image of $\ell^-_1$ under the composition $\gamma^+ \circ \gamma^-$ is the same as the image of $\ell^-_1$ under the monodromy map $\phi$. By the argument above the distance in $\mathcal{C}(F)$ between these loops is at most $12 + 2 + 12 + 2 = 28$. Thus the displacement of $\phi$ is at most $28$. Since we assumed that $d(\phi) > 28$, this contradiction implies that $E^- = D^-$ and $E^+ = D^+$ form the unique reducing pair for $\Sigma$.
\end{proof}

Before we continue to the proof of Theorem~\ref{mainthm}, we note another implication related to the theory of topological index of surfaces. Recall that the \textit{topological index} of $\Sigma$ is zero if $\Sigma$ is incompressible and otherwise is equal to the smallest value $(i+1)$ such that the $i$th homotopy group of the disk set for $\Sigma$ is non-trivial. If the disk complex is contractible then $\Sigma$ will not have a well defined index.

McCullough has shown~\cite{mccullough} that the set of compressing disks on one side of any two-sided embedded surface forms a contractible complex. The disk complex for $\Sigma$ consists of these two contractible sets connected by a single edge between $D^-$ and $D^+$, so we have:

\begin{Coro}
The disk set for $\Sigma$ is contractible.
\end{Coro}

Note that strongly irreducible Heegaard surfaces have index one. A number of weakly reducible Heegaard surface have been constructed with higher indices~\cite{bachman-johnson}\cite{lee}. The surface constructed here appears to be the first class of minimal genus Heegaard surfaces known to have a contractible disk complex. 

We now proceed to the main proof.

\begin{proof}[Proof of Theorem~\ref{mainthm}]
Every automorphism $\psi \in Mod(M, \Sigma)$ takes the weak reducing pair $D^-$, $D^+$ onto a new weak reducing pair of disks. However by Lemma~\ref{mainlem}, $D^-$, $D^+$ is the only weak reducing pair so $\psi$ either fixes these disks (setwise) or transposes them. The surface $F$ results from compressing $\Sigma$ along $D^-$, $D^+$ so we can choose a representative of $\psi$ that takes $F$ to itself. Moreover, because $\psi$ fixes $F \cup \Sigma$ setwise, we can choose a representative that also fixes the arcs $\alpha_\pm$.

By assumption, $\psi \in Isot(M, \Sigma)$ so $\psi$ is isotopic to the identity on $M$ and defines an isotopy from $F$ to itself. The surface $F$ consists of two leaves of the surface bundle defined by $\pi$, each of which has injective fundamental group. The isotopy thus defines an injection from a semi-direct product $\pi_1(F_1) \rtimes \mathbf{Z}$ into $\pi_1(M)$. As in Theoreom 11.1 in~\cite{hem:book}, this implies that the isotopy is induced by a surface bundle structure $\pi' : M \rightarrow S^1$. Because the fiber of $\pi'$ is $F_1$, which is also the fiber of $\pi$, we conclude that $\pi'$ is isotopic to $\pi$ so the isotopy $\psi$ must be defined by the maps $\{r_t\}$ that spin $F$ around the monodromy of the surface bundle.

Each $\psi$ thus determines a half integer and we find a homomorphism from $Isot(M, \Sigma)$ to $\mathbf{Z}$ (or rather to the half integers). If $\psi$ is in the kernel $\mathcal{K}$ of this homomorphism then the restriction of $\psi$ to $F$ is isotopic to the identity (ignoring the endpoints of $\alpha_\pm$). If we consider the endpoints, however, we find a homomorphism from $\mathcal{K}$ to the braid group $B_{2,F_1}$. The kernel $\mathcal{K}'$ of this second homomorphism is the set of all isotopies of $F \cup \alpha_\pm$ that fix the endpoints of $\alpha_\pm$ in $F_1$. The paths of the endpoints in $F_2$, in turn, define a homomorphism into the two-strand braid group on $F_2$. The reader can check that these braids are all homotopy trivial in $F_2$ and that all homotopy trivial braids arise in this way. 

The kernel $\mathcal{K}''$ of this final homomorphism is the set of isotopies of the arcs that fix all four endpoints. Any automorphism of $\Sigma$ induced in this way would be the identity outside a pair of annuli. However, there is a canonical vertical arc transverse to each annulus, so no non-trivial automorphism can arise in this way. Thus $\mathcal{K}''$ is trivial so $\mathcal{K}'$ is isomorphic to the homotopy trivial subgroup of $B_{2,F_2}$. By construction, $\mathcal{K}$ is an extension of $B_{2,F_1}$ by $\mathcal{K}' = B_{2,F_2}$ so $\mathcal{K} = \mathcal{B}$. Finally, $Isot(M, \Sigma)$ is an extension of $\mathcal{K}$ by the integers, with $\mathcal{K}$ a normal subgroup and the integers acting on $\mathcal{K}$ by conjugating by the mondromy map $\phi$. Thus $Isot(M, \Sigma)$ is the semi-direct product $\mathcal{A}$ defined above.
\end{proof}

\bibliographystyle{amsplain}
\bibliography{bundles}

\end{document}